\documentclass[12pt,oneside]{elsarticle}
\biboptions{sort,numbers,longnamesfirst}
\usepackage[english]{babel}
\usepackage{amssymb,amsthm,amsmath}

\usepackage[linkcolor=blue]{hyperref}
\usepackage{amsmath}
\usepackage{amsfonts}
\usepackage{amssymb}
\usepackage{latexsym}
\usepackage{mathrsfs}
\usepackage{euscript}
\usepackage{xspace}
\usepackage{pifont}
\usepackage[capitalize]{cleveref}
\usepackage{refcount}

\usepackage[usenames]{color}
\usepackage{colortbl}
\usepackage{ifthen}

\usepackage[all]{xy}

{\vspace{1em}
  \begin{minipage}{12cm}
    \rule{12cm}{1pt}\sffamily\noindent}%
  {\newline\rule{12cm}{1pt}
  \end{minipage}
  \vspace{1em}}

\renewcommand{\leq}{\leqslant}
\renewcommand{\geq}{\geqslant}

\makeatletter
\newcommand*{\IfItalic}{%
  \ifx\f@shape\my@test@it
  \expandafter\@firstoftwo
  \else
  \expandafter\@secondoftwo
  \fi
}
\newcommand*{\my@test@it}{it}
\makeatother

\DeclareMathOperator{\Aut}{Aut}

\DeclareMathOperator{\ind}{ind}
\DeclareMathOperator{\supp}{supp}

\DeclareMathOperator{\Lin}{Lin}
\DeclareMathOperator{\Comb}{Comb}
\DeclareMathOperator{\Symm}{Symm}

\DeclareMathOperator{\LinZ}{Lin_\Z}

\newtheorem{teo}{Theorem}[section]
\newtheorem{lem}[teo]{Lemma}
\newtheorem{pro}[teo]{Proposition}
\newtheorem{cor}[teo]{Corollary}
\newtheorem{lemma}[teo]{Lemma}

\newcounter{numex}

\newtheoremstyle{exa}{3pt}{3pt}{\small\slshape}{1pt}{\scshape}{:}{1pt}{}
\newtheoremstyle{conj}{3pt}{3pt}{\slshape}{1pt}{\scshape}{:}{1pt}{}
{\theoremstyle{exa}
  \newtheorem{example}{Example}[section]}
{\theoremstyle{conj}
  }

\newtheorem*{Rstat}{\rotulo}
\newcommand{\rotulo}{}
\newenvironment{rstat}[1]
{\renewcommand\rotulo{\cref{#1}}\begin{Rstat}}{\end{Rstat}}

{\begin{enumerate}}%
  {\end{enumerate}}

\newcommand{\R}{\ensuremath{\mathbb{R}}\xspace}
\newcommand{\Q}{\ensuremath{\mathbb{Q}}\xspace}
\newcommand{\Z}{\ensuremath{\mathbb{Z}}\xspace}

\newcommand{\N}{\ensuremath{\mathbb{N}}\xspace}
\newcommand{\one}{\ensuremath{\mathbf{1}}\xspace}
\newcommand{\CC}{\ensuremath{\mathcal{C}}\xspace}

\newcommand{\E}{\ensuremath{\mathcal{E}}\xspace}

\newcommand{\Sn}{\ensuremath{\mathcal{S}_n}\xspace}

\newcommand{\F}{\ensuremath{\mathscr{F}}\xspace}
\newcommand{\Hd}{\(\mathscr{H}\!\!\)-description\xspace}
\newcommand{\Vd}{\(\mathscr{V}\!\!\)-description\xspace}
\newcommand{\op}{\ensuremath{\oplus}\xspace}
\newcommand{\maxclosed}{$\max$-closed\xspace}
\renewcommand{\phi}{\ensuremath\varphi\xspace}
\newcommand{\QN}{\ensuremath{\mathbb{Q}^N}\xspace}
\newcommand{\Qpn}{\ensuremath{\Q_+^n}\xspace}
\newcommand{\RN}{\ensuremath{\mathbb{R}^N}\xspace}

\newcommand{\NN}{\ensuremath{\mathbb{N}^N}\xspace}
\newcommand{\ZN}{\ensuremath{\mathbb{Z}^N}\xspace}
\newcommand{\FN}{\mathbf{N}}
\newcommand{\FT}{\mathbf{T}}
\newcommand{\IN}{N}
\newcommand{\IT}{T}
\newcommand{\iinN}{\ensuremath{{i\in N}}}

\newcommand{\hatex}{\ensuremath{\hat\E_X}\xspace}
\newcommand{\haten}{\ensuremath{\hat\E_n}\xspace}
\newcommand{\ppi}{\ensuremath{P_{\!\pi}}\xspace}

\newcommand{\En}{\ensuremath{\E_n}\xspace}
\newcommand{\Ln}{\ensuremath{\mathscr{L}_n}\xspace}

\newcommand{\conj}[2]{\ensuremath{\{#1\,|\;#2\}}}

\newcommand\ORCiD[1]{\href{https://orcid.org/#1}{\mbox{\scalerel*{
        \begin{tikzpicture}[yscale=-1,transform shape]
          \pic{orcidlogo};
        \end{tikzpicture}
      }{|}}}}
\renewcommand\ORCiD[1]{}

\title{The cone of quasi-semimetrics and\\ exponent matrices of tiled orders}

\author[a1,a3]{Mikhailo Dokuchaev\fnref{f1}}
\ead{dokucha@gmail.com}

\author[a2,a3]{Arnaldo Mandel\corref{co}\fnref{f2}}
\ead{am@ime.usp.br}

\author[a1,a3]{Makar  Plakhotnyk\fnref{f3}} 
\ead{makar.plakhotnyk@gmail.com}

\address[a1]{Mathematics Department}
\address[a2]{Computer Science Department}
\address[a3]{Instituto de Matem\'atica e Estat\'\i stica, Universidade de
  S\~ao Paulo\\
  \emph{\small S\~ao Paulo, SP, Brazil 05508-970}}

\cortext[co]{Corresponding author}

\fntext[f1]{Partially supported by FAPESP of Brazil (Process
  2015/09162-9) and by CNPq of Brazil (Process 307873/2017-0)}

\fntext[f2]{Partially supported by CNPq of Brazil (Proc.
  423833/2018-9).}

\fntext[f3]{Supported by  FAPESP of Brazil (Proc. 2013/11350-2), and partially by CNPq (proc. 101965/2021-4).}

\begin{document}

\begin{abstract}
  Finite quasi semimetrics on $n$ can be thought of as nonnegative
  valuations on the edges of a complete directed graph on $n$ vertices
  satisfying all possible triangle inequalities.  They comprise a
  polyhedral cone whose symmetry groups were studied for small $n$ by
  Deza, Dutour and Panteleeva.  We show that the symmetry and
  combinatorial symmetry groups are as they conjectured.

  Integral quasi semimetrics have a special place in the theory of
  tiled orders, being known as exponent matrices, and can be viewed as
  monoids under componentwise maximum; we provide a novel derivation
  of the automorphism group of that monoid.  Some of these results
  follow from more general consideration of polyhedral cones that are
  closed under componentwise maximum.
\end{abstract}



\maketitle

{
}


\section{Introduction}\label{sec:introduction}


Metric spaces are ubiquitous, making metrics a
well known concept.  Quasi-semimetrics (the term is not totally
standard) are a weakened form of metrics: for a given space \(X\),
\(d(x,y)\) is required to be a nonnegative real, and the triangle
inequality \(d(x,y)\leq d(x,z)+d(z,y)\) has to be satisfied.  It
relaxes two requirements on the definition of a metric, namely, it is
not required to be symmetric, and distinct elements are allowed to be
at ``distance'' zero.  Usually, a space with a fixed metrics or
quasi-semimetrics is studied, for the implied geometrical and
topological structure. There is a plethora of such specific examples in
\cite{DD}.

Here we take a different viewpoint: we fix the space \(X\) and
consider the set of all quasi-semimetrics on it as the main object.
Quasi-semimetrics form a convex cone of real functions on
\(X\times X\); for finite \(X\), that is a rational polyhedral cone,
which we denote \hatex (also \haten if \(X=[n]=\{1,2,\ldots,n\}\)) and
it has been the object of some study (\citet{DezaDutourPant},
\citet{DezaDezaDut}, \citet{DezaPant} and \citet{DDS}, which denote
it as \emph{QMET\({}_n\)}).  It is convenient, for the discussion
below, to have the explicit description of \haten as the subset of
\(M_n(\R)\) consisting of matrices \(X=(x_{ij})\) such that for all
pairwise distinct \(i,j,k\),
\begin{equation}\label{hatensys}
  \begin{array}{lrl}
    \IT_{ijk}:           & \; x_{ij}+x_{jk}            & \geq x_{ik}\,, \\
    \IN_{ij\phantom{k}}: & \; \phantom{x_{jk} +}x_{ij} & \geq 0\, ,     \\
    {}                 & \; \phantom{x_{ij}+}x_{ii}  & =0\, .
  \end{array}
\end{equation}

It is also convenient to think of such \(X\) as
an assignment of values to the edges of the complete directed graph with vertex set \([n]\).
Following \cite{DezaDezaDut}, the \(\IT_{ijk}\) are called
\emph{triangle inequalities}, while \(\IN_{ij}\) are
\emph{nonnegativity inequalities}. Also, to avoid special cases
requiring definition acrobatics, we stipulate that \(n\geq 3\),
whenever \haten is considered.

This family of cones is quite thoroughly discussed in
\cite{DezaDezaDut}, and one aspect will be relevant here: the
determination of their (Euclidean) symmetry group and combinatorial
symmetry group (those are defined in \cref{preliminaries}).

Given the description of \haten, there is a natural class of linear maps that
preserve the cone: permutations of coordinates that leave the system
describing \haten invariant.  Those are of two types:
\begin{enumerate}
    \item Any permutation \(\pi\) on the indices induces the permutation
  \(\ppi: x_{ij}\mapsto x_{\pi(i)\pi(j)}\).
    \item The transpose map (called reversal in \cite{DDS}) \(\tau\):
  \(x_{ij}\mapsto x_{ji}\).
\end{enumerate}
We call those \emph{system automorphisms}, and denote the group they form by
\Sn.  Since \(\tau\) commutes with all permutations of the first type, it
follows that \(\Sn\cong S_n \times \mathbb{Z}_2\).

These maps are isometries, so they are symmetries of cone; they
naturally induce a subgroup of the combinatorial automorphism group of
\haten, which we also will denote as \Sn.  \citet{DezaDezaDut}
verified computationally that \Sn is the whole symmetry and
combinatorial symmetry group of \haten for small \(n\) and conjectured
(also in \cite{DDS}) that this was the case in general.  Our main
result here settles those conjectures:

\begin{rstat}{comb-auto}
  The combinatorial symmetry group of \haten is \Sn.
\end{rstat}

As every cone, \hatex is a semigroup under addition; it is, moreover,
closed on an additional operation, which also turns it into a
commutative semigroup: componentwise maximum.  It is convenient to use
the infix notation \(a\oplus b=\max(a,b)\), for real \(a,b\), and
extend the notation componentwise to real vectors or matrices: if
\(u,v\in\R^I\), \(u\oplus v\) is defined by
\((u\oplus v)_i=u_i\oplus v_i\); we refer to this operation simply as
\(\max\).  These two operations (and the addition of a \(-\infty\))
turn \(\R^X\) into a semiring, a coproduct of copies of the tropical
semiring.  In this context, \hatex is a subsemiring of \(\R^X\), but 
not, however, a subsemimodule.

We will be especially interested in integer valued quasi-semimetrics,
and denote \(\En=\haten\cap\Z^n\).  There are many reasons to
concentrate on the integer points in a rational cone (see \cite{BG},
\cite{BR}), and in this particular case they have appeared in a
context far removed from the usual study of polyhedra, the theory of
tiled orders in algebras.  Those are described based on a discrete
valuation ring and a matrix, which, by \cite[Lemma 1.1]{jate2}, is an
integer valued quasi-semimetrics; in that context, those have been
called \emph{exponent matrices}.  We refer to \cite{jate2},
\cite{DArnold}, \cite{DemonetLuo}, \cite{DKKP} for definitions, more
details and some applications; we will not mention tiled orders any
more here, but honor them in the notation \hatex.

Clearly \Sn respects integrality and the operations of addition and
\(\max\), so it also restricts to \((\En,+)\) and \((\En,\op)\)
automorphisms.  A natural question is whether new symmetries or
automorphisms can crop up if we consider each of those monoids.

That is not the case: it was proved in \cite{DKKP} that the
automorphism group of \((\En,+)\) and \((\En,\op)\) is \Sn. 
The proofs were considerably ad-hoc and elaborate, and
here we will present new, and somewhat more conceptual proofs.

Any automorphism of \((\En,+)\) naturally extend to a linear
automorphisms of \haten, thus respecting its face-lattice, as well as
any symmetry of \haten does.  So, our main theorem implies that both
groups coincide with \Sn (we elaborate on that in
\cref{preliminaries}).

For some perspective on this result, we refer the reader to
\cite{BSPRS}; there, some elegant algorithms for computing the
symmetry and combinatorial symmetry group of a given polyhedral cone
are shown; both are reduced to finding the automorphism group of a
colored graph.  There is a difference, though.  The graph for the
symmetry group has as vertices the extreme rays or the facets of the
cone, so it has polynomial size in the description of the cone.  For
the combinatorial symmetry group, the graph is the incidence graph of
extreme rays and facets; this may have exponential size relative to a
given description, and this is indeed the case with \haten.  This
method underlies our proof of the main theorem, but we were able to
finesse the problem of describing the extreme rays by showing that
just a small part of them suffices.

The automorphism group of \((\En,\oplus)\) are derived from a more general
view of cones closed under \(\max\), and their integer point submonoids.  In
this case, in general, a \op-automorphism does not need to respect the
face-lattice (\cref{op-nonrespecting}), and may even not be extendable to the
whole cone.  We present some conditions which imply that the automorphisms of
the integer submonoid of a \maxclosed cone are induced by permutations of the
coordinates.  This is enough to show that \(\Aut(\En,\oplus)=\Aut(\En,+)\)

The article proceeds as follows: We start by recalling some basics on
polyhedral cones and proving some initial facts in
\cref{preliminaries}.  \cref{sec:addcomb-autom} proves the main
theorem.  \cref{sec:max-cones} presents some basic facts about
\maxclosed cones.  This is followed by \cref{sec:op-automorphisms}
where we prove the aforementioned result on \op-automorphism of the
integer submonoid of a \maxclosed cone, entailing, in particular, that
\(\Aut(\En,\op)=\Sn\).

\section{Preliminaries on polyhedral cones}\label{preliminaries}

We present here a summary of facts and terminology about polyhedral cones;
some of these have been appropriately streamlined for our needs.
For more detailed information and proofs the reader is referred to
\cite{BG,Sc,DDS}.  Besides the definitions, we make several assertions
about cones without further ado; they are well-known facts that can be
found in the references, and are easy exercises.  As the cones we are interested in
are the \haten, we illustrate the concepts as they directly apply to them.

In what follows, \(I,J, N\) will denote finite sets; \R, \Z, \(\R_+\), \N will
stand for the sets of real numbers, integers, non-negative reals and
non-negative integers, respectively. In the vector space \RN\ we single out
the \emph{canonical basis vectors} \(e_i\) and the \emph{one vector}
\(\one=\sum_ie_i\); on \RN, \(x\leq y\) means \(x_i\leq y_i\) for all
\(i\in N\), and \(x\geq y\) means \(y\leq x\).  The \emph{support} of
\(v\in\RN\) is \(\supp(v)= \conj{i\in N}{v_i\neq 0}\) and its cardinality
will be denoted \(s(v)\).  A subset of \RN is \emph{full-dimensional} if it
linearly spans the whole space.
We will consider subsets \(S, S'\) of \RN\ to be equivalent if there is a
bijection between \(S\) and \(S'\) such that the image of each vector is a
positive scalar multiple of it. A \emph{ray} is an equivalence class of a nonzero
singleton, and we will say that \(S\) is \emph{clean} if its elements belong
to different rays.  A point (or ray) \(x\) that satisfies a linear inequality
\(ax\geq 0\), does it \emph{exactly} if \(ax=0\) and \emph{strictly} if
\(ax>0\). A ray is \emph{rational} if it contains a vector with rational
coordinates.

\begin{example}\label{vector-space-haten}
  As defined, the cone \haten lies in the subspace of \(n\times n\)
  real matrices with null diagonal; it is convenient to consider this
  subspace to be the whole ambient space.  So, for a fixed \(n\), let
  \(N_n=\conj{(i,j)}{1\leq i,j\leq n, i\neq j}\), and we take the space
  \(\R^{N_n}\) to be the one where \haten is defined.  The vectors
  in this space are still better visualized (and referred to) as matrices
  with a blotted diagonal, rather than a linear list of coordinates.
  Written in the format \(ax\geq 0\), the defining inequalities
  \(\IT_{ijk}\) take the form \(x_{ij}+x_{jk}-x_{ik}\geq 0\), whose
  coefficient vector \(a\) has support of size 3.
\end{example}


A finite set of non-zero vectors \(S\) is said to be a \emph{\Vd} of the set
\(\conj{\sum_{v\in S}\lambda_v
  v}{\lambda_v\in\R_+\ \text{for all}\ v\in S}\), and $S$ is also called
an \emph{\Hd} of
\(\conj{x\in \RN}{v^tx\; \geq 0\ \text{for all}\ v\in S}\).  We may think
of \(S\) as the set of rows of a matrix \(A\); then \(S\) is an \Hd of
\conj{x\in\RN}{Ax\geq 0}.  The \emph{Weyl-Minkowski Theorem}
(see~\cite{Sc} and~\cite{BG}) states that a set has a \Vd\ if and only if it
has an \Hd; further, it has a \Vd with rational rays if and only if it has an
\Hd in which the matrix has only rational entries.  A set with either
description is called a \emph{polyhedral cone}, and it is a
\emph{rational polyhedral cone} if it has either description using only rational
data.  Clearly, equivalent sets describe the same cones; either
way, just clean descriptions suffice.

A cone \CC\ is \emph{pointed} if the only linear subspace it contains is
\((0)\).  A cone \CC is full-dimensional if there is a point that satisfies
all inequalities of an \Hd strictly.

\begin{example}\label{haten-as-cone}
  \haten was defined by inequalities; that is, we have an explicit \Hd
  of \haten, the corresponding set \(S\) consisting of coefficient
  vectors (matrices, actually) of those inequalities.  The
  coefficients are just \(0,1,-1\), which shows that \haten is a
  rational polyhedral cone, and also that the description is clean.
  Moreover, each inequality induces a facet (a concept defined below
  and a fact proved in \cref{sec:addcomb-autom}).  The matrix \one
  satisfies all inequalities strictly, showing that \haten is
  full-dimensional; the nonnegativity inequalities easily imply that
  \haten is pointed.

  The Weyl-Minkowski Theorem implies that \haten also has a \Vd.
  Describing it explicitly is a possibly impossible task.  In
  \cite{DezaDutourPant} and \cite{DDS} there are explicit descriptions
  of the rays for \(n\leq 4\).  After that, there are descriptions of
  some families and some computational results.  The number of rays
  grows exponentially with \(n^2\), so computations quickly stop
  short.  However, a small family of rays described in
  \cref{sec:addcomb-autom} will be crucial in the proof of the main
  theorem.
\end{example}

A linear inequality \(ax\geq 0\), with \(a\neq 0\) that holds for
every \(x\in\CC\) is a \emph{valid inequality} for \CC; the
\emph{face} of \CC\ it \emph{induces} is the set
\conj{x\in\CC}{ax = 0}.  We also consider \CC a (improper) face. The
faces of a cone, ordered by inclusion, comprise a lattice, the
\emph{face-lattice} of the cone, with intersection as the meet
operation.  The face lattice is finite and graded.  A \emph{facet} is
a maximal proper face. If \CC\ is full-dimensional, every facet is
induced by a unique (up to equivalence) inequality, and the collection
of such \emph{facet-inequalities} comprises the unique minimal \Hd\ of
\CC.  The face-lattice is \emph{coatomistic}, that is, every proper
face is an intersection of facets; equivalently, in any \Hd, a face is
a subset of \CC that satisfies some fixed subset of the inequalities
exactly.  A point is \emph{interior} to a face if the valid
inequalities it satisfies exactly are precisely those that induce the
face; every face has an interior point.  Equivalently, a point \(p\)
is interior to a face \(F\) if and only if the facets containing \(p\)
are those that contain \(F\); in particular, in a clean \Hd, a
facet-inequality is one such that there is a point satisfying that one
exactly, and all other inequalities strictly.  It also follows that
\CC is full-dimensional if and only if it has an interior point.  If
the cone is pointed, the minimal non-zero faces are rays, so called
\emph{extreme rays}, and these comprise the unique minimal \Vd\ of the
cone.  The face-lattice is also \emph{atomistic}: every face is a join
of extreme rays.

An \emph{integer cone} is the intersection \(\CC_{\Z}\) of a rational cone \CC
with \(\Z^N\).  Such a cone is an additive submonoid of \(\Z^N\), and it is
finitely generated.  If \CC\ is pointed, there exists a unique minimal set of
generators, called a \emph{Hilbert basis}, and it is finite (see
\cite[Theorem 16.4]{Sc}, \cite[Chapter 2]{BG}); it contains one point in each extreme ray, and
usually some more points.  


The presentation above relies on a fixed system of coordinates, given
by the basis of elementary vectors.  A more elegant, coordinate free
approach is used in \cite{BG}, and it gives an account of all relevant
concepts related to the face-lattice.  However, working with a fixed
basis comes naturally when handling systems of linear inequalities;
moreover, the \(\max\) operation is naturally and traditionally
defind based on coordinates.




A \emph{linear automorphism} of a cone \(\CC\subseteq\RN\) is a linear
automorphism \(\varphi\) of \RN such that \(\varphi(\CC)=\CC\).  If
\(\varphi\) is an isometry preserving Euclidean distance, it is said
to be a \emph{isometry} of \CC (\cite{DDS} calls it a
\emph{symmetry} of \CC). It is clear from the definition that any
linear automorphism of \CC maps faces to faces, and induces an
automorphism of the face lattice of \CC; in particular, the
families of extreme rays and of facets are invariant.

We single out four symmetry groups associated with a given cone \CC (we
combine the notation of \cite{BSPRS} and \cite{DDS}, with occasional slight
change of meaning):
\vspace{-1ex}
\begin{itemize}
    \item \(\Comb(\CC)\) - the \emph{combinatorial symmetry group},
  consisting of all automorphisms of the face-lattice of \CC.
    \item \(\Lin(\CC)\) - the \emph{linear symmetry group},
  consisting of all linear automorphisms of \CC.
    \item \(\Symm(\CC)\) - the \emph{symmetry group},
  consisting of all isometries of \CC (named as in  \cite{DDS}).
    \item \(\LinZ(\CC)\) - the \emph{integral symmetry group},
  consisting of all linear automorphisms leaving \(\CC\cap \ZN\) invariant.
\end{itemize}

So, any linear automorphism of \CC induces an automorphism of the face
lattice of \CC, and this induction is indeed a group homomorphism
\(\ind\!:\!\Lin(\CC)\rightarrow\Comb(\CC)\).  As the face-lattice is both
atomistic and coatomistic, both the set of facets and the set of extreme rays
are bases for the permutation group \(\Comb(\CC)\); that is, any element of
this group is fully determined by its action on either set.  The approach
favored in \cite{BSPRS} is to represent automorphisms by their action on
extreme rays, while in \cref{sec:addcomb-autom} we find it convenient to
represent them by their action on the facets.  After all, convenience depends
on the available description of the cone.

\begin{pro}\label{injmodcomb}
  If \CC is a full dimensional pointed rational cone, the restrictions of
  \(\ind\) to \(\Symm(\CC)\) and \(\LinZ(\CC)\) are injective.
\end{pro}
\begin{proof}
  Consider first the restriction to \(\LinZ(\CC)\). If \(\varphi\) is in
  the kernel of this map, it leaves each ray invariant.  But the set of
  integral vectors in the ray is also invariant, and that implies that
  \(\varphi\) is the identity on that ray.  So it is the identity
  automorphism.  For \(\Symm(\CC)\), we apply the same argument to the unit
  vector in each ray.
\end{proof}

\begin{example}
  Notice that, in spite of the similarity exposed in the proof of
  \cref{injmodcomb}, \(\Symm(\CC)\) and \(\LinZ(\CC)\) can be quite
  different. Consider the cone
  \(\CC_1=\conj{x\in\R^2}{x_2\geq 0, x_1-x_2\geq 0}\); the map
  given by {\tiny\(\left(
      \begin{array}[h]{rr}
        1 & 0  \\
        1 & -1
      \end{array}
    \right)
    \)}
  is in  \(\LinZ(\CC_1)\) but not in \(\Symm(\CC_1)\).  On the other hand, for
  \(\CC_2=\conj{x\in\R^2}{x_2\geq 0, 3x_1-4x_2\geq 0}\), the map given by
  {\tiny\(\frac15\left(
      \begin{array}[h]{rr}
        4 & 3  \\
        3 & -4
      \end{array}
    \right)
    \)}
  is in  \(\Symm(\CC_2)\)  but not in \(\LinZ(\CC_2)\).
\end{example}

The restrictions of \(\ind\) above (actually, \(\ind\)
itself) can be far from surjective.  To see this, take your favorite highly
symmetric cone and apply to it a linear transformation that is neither
orthogonal nor integral.  The poor image's symmetry and integral symmetry groups
becomes severely handicapped, while the combinatorial symmetry group gets away
scot-free.

The astute reader may complain that this is a trick; a deeper
construction of \citet{BEK} presents a polytope lattice and a
combinatorial symmetry that cannot be realized linearly for any
polytope with that face-lattice.  A standard construction turns 
this into a result about cones.

Going back to the cone \haten, we notice that \Sn consists of maps that are
both isometries and integral, that is, we have the following diagram of
monomorphisms:
\begin{figure}[h]
  \[
    \footnotesize
    \xymatrix{&  & \Symm(\haten)\quad\quad \ar@<4pt>[dr]^\ind & \\
      \Sn \ar[r]^-\subseteq &\Symm(\haten)\cap\LinZ(\haten)
      \qquad\qquad\qquad\ar[ur]^\subseteq\ar@<4pt>[dr]^\subseteq &  & \Comb(\haten) \\
      & & \LinZ(\haten)\quad \ar@<4pt>[ur]^\ind &
    }
  \]
  \caption{}
  \label{eq:SCinter}
\end{figure}

As it turns out in \cref{allequal}, the composition
\(\Sn\rightarrow \Comb(\haten)\) is surjective, hence all inclusions are
equalities.

\section{Symmetries and combinatorial automorphisms of \haten}\label{sec:addcomb-autom}

The main result in this section is \cref{comb-auto}, which describes
the combinatorial automorphism group of \haten.  For \(n\leq 5\) this
has been done in \cite{DezaDutourPant}, computationally.  As
noted before, while some shortcuts exist, the general method for
computing the combinatorial automorphism group of a cone is to
determine the bipartite incidence graph of extreme rays and facets,
and then computing the automorphism group of the graph.  Although no
polynomial algorithm is yet known for such computation, there are very
good programs \cite{MKP} that can handle graphs of fairly large size.

As per \cref{facet-ineq}, \haten has \(n^2(n-1)/2\) facets.  However,
\cite{DezaDutourPant} tells us that for \(n\geq 6\), the number of extreme
rays is already too big for polite computational society.

As it turns out, there is an orbit of \Sn, denoted \Ln, consisting of
\(2n\) extreme rays, such that it is enough to consider the incidence graph
of facets and \Ln to clinch \(\Comb(\haten)\); the following nice
properties hold:
\begin{enumerate}
    \item \Ln is an orbit of \(\Comb(\haten)\) (\cref{LCinv}).
    \item The action of \(\Comb(\haten)\) on \Ln is the same as the action of \Sn (\cref{RC-ind}).
    \item The action above is faithful.
\end{enumerate}
The last item is what will finally establish the main theorem.

Recall that \(\IT_{ijk}\) and \(\IN_{ij}\) are the defining
inequalities for \haten; in what follows, the same labels in boldface
(\(\FT_{ijk}\), etc.) will denote the corresponding faces of \haten,
which turn out to be facets of \haten.



The following facts about some special members of \haten appear in part in
\cite[Theorem 5]{DezaDutourPant}, and in \cite[Theorem 1.1]{DKKP}.  We present
them here with proofs, for completeness.  For each proper subset \(I\) of
\(\{1,\ldots,n\}\), the associated \emph{oriented cut quasi-semimetrics}
(\cite{DezaPant,DezaDutourPant,DD}) is the binary exponent matrix \(\delta(I)\)
such that \(\delta(I)_{ij}=1\) if and only if \(i\in I,j\notin I\).

\begin{pro}\label{cut-semi}
  Considering \haten:
  \begin{enumerate}
      \item The oriented cut quasi-semimetrics are those with minimal nonempty supports.
      \item If \(A\in\haten\) and \(\supp(A)=\supp(\delta(I))\), then \(A\) is a
    scalar multiple of \(\delta(I)\).
      \item All oriented cut quasi-semimetrics are extreme rays of \haten.
  \end{enumerate}
\end{pro}
\begin{proof}
  For (a), let \(A\in\haten\), and suppose \(A_{rs}>0\).  Let
  \(I=\conj{k}{A_{rk}=0}\); this is a proper subset of indices, as \(r\in I\),
  \(s\notin I\). Then, if \(i\in I, j\notin I\), \(A_{ri}= 0 \neq A_{rj}\), and
  \(\IT_{rij}\) implies \(A_{ij}>0\).  It follows that
  \(\supp(\delta(I))\subseteq\supp(A)\).  For (b), let \(i\in I\) and suppose
  there exist distinct \(j,k\notin I\).  Applying \(\IT_{ijk}\) and \(\IT_{ikj}\) we
  conclude that \(A_{ij}=A_{jk}\); that is, all nonzero terms on each row of \(A\)
  are equal; the same argument applies to columns.  So, all nonzero elements of
  \(A\) are equal, and the result follows. Finally, for (c), let \(A\) be an
  interior point of the minimal face containing \(\delta(I)\).  It must satisfy
  exactly the same inequalities as \(\delta(I)\); in particular, the same
  \(\IN_{ij}\), hence \(\supp(A)=\supp(\delta(I))\).  From part (b), it follows that
  the face has dimension 1.
\end{proof}


This is a technical workhorse for what follows:

\begin{lem}\label{faces-inter}
  For all three distinct indices \(i,j,k\) the only defining inequalities
  of \haten exactly satisfied by \(\FN_{ij} \cap \FN_{jk}\) are \(\IN_{ij}, \IN_{jk},\IN_{ik}\)
  and \(\IT_{ijk}\).
\end{lem}
\begin{proof}
  If \(x\in \FN_{ij} \cap \FN_{jk}\), then \(x_{ij}=x_{jk}=0\), and \(\IT_{ijk}\)
  implies that \(x_{jk}\leq 0\).  This implies that \(x\) satisfies both
  \(\IN_{ik}\) and \(\IT_{ijk}\) exactly.

  In order to show that no other inequality is satisfied exactly, we construct
  an exponent matrix \(H=H(i,j,k)\) as follows. For distinct \(r,s\),
  \[
    H_{rs}=
    \begin{cases}
      0 & \; \text{if}\; rs= ij,jk, ik,                                                                   \\
      3 & \; \text{if}\; rs= ji,kj\quad \text{or}\quad r=i, s\neq j,k\quad \text{or}\quad s=k, r\neq i,j, \\
      4 & \; \text{if}\; rs= ki\quad \text{or}\quad r=j, s\neq i,k\quad \text{or}\quad s=j, r\neq i,k,    \\
      5 & \; \text{otherwise}.
    \end{cases}
  \]
  See \cref{fig:h123} for an illustration.
  \begin{figure}[htb]
    \centering
    \[
      \left(
        \begin{array}{ccccccc}
          0          & 0          & 0          & 3          & 3          & 3 & \rightarrow \\
          3          & 0          & 0          & 4          & 4          & 4 & \rightarrow \\
          4          & 3          & 0          & 5          & 5          & 5 & \rightarrow \\
          5          & 4          & 3          & 0          & 5          & 5 & \rightarrow \\
          5          & 4          & 3          & 5          & 0          & 5 & \rightarrow \\
          5          & 4          & 3          & 5          & 5          & 0 &             \\
          \downarrow & \downarrow & \downarrow & \downarrow & \downarrow &   &
        \end{array}
      \right)
    \]
    \caption{\(H(1,2,3)\). Arrows mean repeat the term in that direction.}
    \label{fig:h123}
  \end{figure}

  It is quite clear that the only nonnegativity inequalities satisfied by
  \(H\) are \(\IN_{ij}, \IN_{jk},\IN_{ik}\), and it also satisfies \(\IT_{ijk}\)
  exactly.  To see that \(H\in\haten\), as well as that it does not satisfy
  any other triangle inequality exactly can be done by case analysis.
  Separating the remaining \(\IT_{rst}\) according with \(r=i,r=j,r=k\),
  \(s=i,s=j,s=k\) (some of these cases are not mutually exclusive), and all	
  remaining cases, leads to an easy verification that \(H\) satisfies all these
  \(\IT_{rst}\), no one exactly.
\end{proof}

The fact below is proved in \cite{DDS} as a consequence of a method of lifting
facets from \haten to \(\hat\E_{n+1}\); that kind of obscures its
simplicity.
\begin{cor}\label{facet-ineq}
  All nonnegativity and triangle inequalities are facet defining for \haten.
\end{cor}
\begin{proof}
  \cref{faces-inter} implies that no face induced by each nonnegativity
  inequality is contained in any other face, so they all induce facets.  But
  clearly \(\haten\neq \R^n\), so, at least one triangle inequality is
  facet-inducing.  As the group \Sn acts transitively on the set of triangular
  inequalities, they all induce facets as well.
\end{proof}

Here is a more direct proof:
\begin{quote}
  The matrix whose \(ij\)-entry is \(0\) if \(ij=rs\), \(2\) if \(i=s\) or
  \(j=r\) and \(1\) otherwise is an interior point of \(\FN_{rs}\).  The matrix
  whose \(ij\)-entry is \(1\) if \(ij=rs\) or \(st\) and is \(2\) otherwise is
  an interior point of \(\FT_{rst}\).  Each of these facts can be verified by
  inspection.
\end{quote}

The oriented cut quasi-semimetrics associated to sets of size \(1\) and \(n-1\)
play a very special role, already detected in \cite{DKKP}.  Denote
\(R^{(r)}=\delta(\{r\})\), \(C^{(r)}=\delta(\{1,\ldots,n\}\backslash\{r\}\).
In matrix form, \(R^{(r)}\) has row \(r\) with all ones off
diagonal, and is zero elsewhere; ditto for \(C^{(r)}\) and column \(r\), so that
\(C^{(r)}=\tau(R^{(r)})\).  We denote
\(\Ln=\conj{R^{(r)},C^{(r)}}{1\leq r\leq n}\) and refer to its members as
\emph{lines}.  As noted before, all rays in \Ln are extreme in \haten; an
alternative proof is in \cite[Lemma 3.2]{DKKP}.

Recall that \(s(A)\) is the size of the support of \(A\), that is the number
of nonzero entries in \(A\).
\begin{pro}\label{min-support}
  If \(0\neq A\in\haten\), then \(s(A)\geq n-1\).  If
  \(s(A)=n-1\), then it is a multiple of some line.
\end{pro}
\begin{proof}
  This can be read directly from \cref{cut-semi}, as \(s(\delta(I))=|I|(n-|I|)\).
\end{proof}

Direct inspection shows that:

\begin{pro}\label{RC-eqs}
  Recall that when mentioning \(\IT_{ijk}\) and \(\IN_{ij}\) all indices are
  distinct.
  \begin{enumerate}
      \item \(R^{(r)}\) satisfies exactly only \conj{\IT_{ijk}}{j\neq r } and
    \conj{\IN_{ij}}{i\neq r}.
      \item \(C^{(s)}\) satisfies exactly only \conj{\IT_{ijk}}{j\neq s} and
    \conj{\IN_{ij}}{j\neq s}.
      \item \(\IN_{ij}\) is satisfied exactly on \Ln only by those \(R^{(r)}\)
    such that \(r\neq i\) and those \(C^{(s)}\) such that \(s\neq j\).
      \item \(\IT_{ijk}\) is satisfied exactly on \Ln only by those \(R^{(r)}\)
    such that \(r\neq j\) and those \(C^{(s)}\) such that \(s\neq j\).
  \end{enumerate}
\end{pro}

For each facet \(F\), let
\(e(F)=\conj{R\in\Ln}{R\text{ satisfies the \(F\) inequality exactly}}\).  If
we think of the members of \Ln as rays,
\(e(F)=\conj{R\in\Ln}{R\subseteq F}\).

\begin{pro}\label{nonneg-inj}
  If for some \(\FN_{ij}\), \(e(F)=e(\FN_{ij})\), then \(F=\FN_{ij}\).
\end{pro}
\begin{proof}
  From \cref{RC-eqs}, \(e(\FN_{ij})=\Ln\backslash\{R^{(i)},C^{(j)}\}\), while
  \(e(\FT_{ijk})=\Ln\backslash\{R^{(j)},C^{(j)}\}\), and the result follows.
\end{proof}

We now state the main result in this section.

\begin{teo}\label{comb-auto}
  The combinatorial automorphism group of \haten consists precisely of those
  permutations of the face-lattice induced by \Sn, that is, the restriction of
  \(\ind\) to \Sn is an isomorphism.
\end{teo}

In view of \cref{injmodcomb}, we have

\begin{teo}\label{allequal}
  \(\Symm(\haten)=\LinZ(\haten)=\Sn\).
\end{teo}

Some of these results are proved in \cite{DDS} for small values of \(n\) and
conjectured to hold in general.

\begin{proof}[Proof of \cref{comb-auto}]
  By (\cref{eq:SCinter}) and \cref{injmodcomb}, it is enough to show that
  \(\ind(\Sn)=\Comb(\haten)\).  The set \F of facets is invariant under
  \(\Comb(\haten)\); since every face of \haten is a meet of facets, it is
  enough to show that for every \(\varphi\in\Comb(\haten)\) there exists a
  \(\psi\in\Sn\) whose action on \F coincides with that of \(\varphi\).

  Some Lemmas below pave the way to \cref{RC-ind}, which shows \Ln is
  invariant under \(\Comb(\haten)\), and each combinatorial automorphism acts
  on \Ln in the same way as some automorphism induced from \Sn.

  To finish the proof, let \(\varphi\in\Comb(\haten)\), and \(\psi\) be given
  by \cref{RC-ind}; then
  \(\gamma=(\ind\psi)^{-1}\varphi\in\Comb(\haten)\) is the identity on \Ln.

  As \(\gamma\) is a combinatorial automorphism, it commutes with
  \(e\). Hence, given a nonnegativity facet \(N_{ij}\),
  \(e(\gamma(N_{ij}))=\conj{\gamma(R)}{R\in e(N_{ij})}=e(N_{ij})\), and it
  follows from \cref{nonneg-inj} that \(\gamma(N_{ij})=N_{ij}\); 
  \cref{faces-inter} implies that \(\gamma\) fixes the triangular facets as
  well. Hence \(\gamma\) is the identity, and \(\varphi=\ind\psi\).
\end{proof}

Those were the facts used in the proof of  \cref{comb-auto}:
\begin{lem}\label{nstrict}
  Any \(A\in\haten\) strictly satisfies at least \((n-2)p(A)+s(A)\) defining
  inequalities, where \(p(A)\) is the number of pairs \(\{i,j\}\) such that at
  least one of \(a_{ij},a_{ji}\) is positive.
\end{lem}
\begin{proof}
  If \(A\) satisfies both \(\IT_{ijk}\) and \(\IT_{jik}\) exactly, then
  \(a_{ij}+a_{ji}=0\), hence these entries are \(0\). So, if at least one of
  \(a_{ij},a_{ji}\) is positive, at least one of \(\IT_{ijk}\) and \(\IT_{jik}\)
  is strict for \(A\), for each \(k\).  That gives \((n-2)p(A)\) strict inequalities.
  The  number of strict nonnegativity inequalities is \(s(A)\).
\end{proof}

\begin{lem}\label{LCinv}
  The set \Ln is invariant under any combinatorial automorphism of \haten.
\end{lem}
\begin{proof}
  \cref{RC-eqs} implies that the ray spanned by each \(R^{(r)}\) and \(C^{(s)}\)
  is contained in \((n-1)^2(n-2)\) triangular facets and \((n-1)^2\) nonnegativity
  ones, so it is contained in precisely \((n-1)^3\) facets.  We will show that any
  other nonzero face is contained in fewer facets.  This immediately implies the
  result.

  Let \(F\) be a nonzero face of \haten and let \(A\) be an interior point of
  \(F\).  As there exist \(n(n-1)(n-2)\) triangular facets and \(n(n-1)\)
  nonnegativity facets, for a total of \(n(n-1)^2\), we want to show that if
  \(F\notin\Ln\) then \(A\) strictly satisfies more than
  \((n-1)^2=n(n-1)^2-(n-1)^3\) inequalities.

  By \cref{cut-semi}, we need to consider only two cases
  \begin{enumerate}
      \item There exists a subset \(I\) of \NN such that \(2\leq |I|\leq n-2\)
    and for every \(i\in I, j\notin I\), \(a_{ij}>0\). Referring to
    \cref{nstrict}, both \(p(A)\) and \(s(A)\) are at least \(|I|(n-|I|)\);
    so \(A\) strictly satisfies at least \((n-1)|I|(n-|I|)\) inequalities.  As
    \(2\leq |I|\leq n-2\), \(|I|(n-|I|)>n-1\), and the result follows.
      \item There exist two members of \Ln such that the support of \(A\)
    contains the union of their supports, hence \(s(A)\geq 2n-3\).  Trivially,
    \(p(A)\geq n-1\), so \(A\) strictly satisfies at least
    \((n-2)(n-1)+2n-3>(n-1)^2\) inequalities.
  \end{enumerate}
\end{proof}

Denote \(\mathscr{R}=\conj{R^{(i)}}{i=1,\ldots,n},\, \mathscr{C}=\conj{C^{(i)}}{i=1,\ldots,n}\).
\begin{lem}\label{RC-part}
  The partitions \((\mathscr{R},\mathscr{C})\) and
  \(\left(\left\{R^{(i)},C^{(i)}\right\}_{i=1,\dots,n}\right)\) of \Ln are preserved by
  any combinatorial automorphism of \haten.
\end{lem}
\begin{proof}
  We denote by \(c(A,B)\) the number of facets containing lines \(A,B\); this is
  clearly a combinatorial invariant.  Let us compute these numbers
  using \cref{RC-eqs} (a) and (b).  There are three cases to consider:
  \begin{enumerate}
      \item \(A=R^{(r)},B=R^{(s)}\) or \(A=C^{(r)},B=C^{(s)}\), \(r\neq s\).
    We count the \(\IT_{ijk}\) with \(j\neq r,s\) and \(\IN_{ij}\) with
    \(i\neq r,s\) (case \(R\)) or \(j\neq r,s\) (case \(C\)). Hence\\
    \(c(A,B)=(n-1)(n-2)^2+(n-1)(n-2)=(n-1)^2(n-2)\).
      \item \(A=R^{(r)},B=C^{(s)}\), \(r\neq s\). We count the \(\IT_{ijk}\)
    with \(j\neq r,s\) and \(\IN_{ij}\) with \(i\neq r\) and \(j\neq s\).  Hence
    \(c(A,B)=(n-1)(n-2)^2+(n-1)(n-2)+1=(n-1)^2(n-2)+1\).
      \item \(A=R^{(r)},B=C^{(r)}\). We count the \(\IT_{ijk}\) with \(j\neq r\) and \(\IN_{ij}\) with
    \(i\neq r\) and \(j\neq r\). Hence
    \(c(A,B)=(n-1)^2(n-2)+(n-1)(n-2)=n(n-1)(n-2)\).
  \end{enumerate}
  Consider now the complete graph with vertex set \Ln and edges colored by the
  values of \(c(A,B)\) just computed.
  By \cref{LCinv}, any combinatorial automorphism of \haten permutes the
  vertices of the graph, and, since the colors are
  combinatorial invariants, such permutation is an automorphism of the
  colored graph.
\end{proof}

\begin{lem}\label{RC-ind}
  If \(\varphi\) is a combinatorial automorphism of \haten, then there exists
  \(\psi\in\Sn\) such that for every \(R\in\Ln\),
  \(\varphi(R)=(\ind\psi)(R)\).
\end{lem}
\begin{proof}
  Since \Ln is invariant,
  \(\varphi\left(R^{(1)}\right)=R^{(j)}\text{ or }C^{(j)}\), for some \(j\).
  Consider first the case \(\varphi\left(R^{(1)}\right)=R^{(j)}\).  It follows
  from \cref{RC-part}, \(\varphi(\mathscr{R})=\mathscr{R}\), hence there
  exists \(\pi\in S_n\) such that for every \(i\),
  \(\varphi\left(R^{(i)}\right)=R^{(\pi(i))}\).  Also, for every \(i\),
  \(\varphi\left(\left\{R^{(i)},C^{(i)}\right\}\right)=\left\{R^{(\pi(i))},\varphi\left(C^{(i)}\right)\right\}\),
  and again by \cref{RC-part}, we have that
  \(\varphi\left(C^{(i)}\right)=C^{(\pi(i))}\), and the result follows, with
  \(\psi=\ppi\).  In the case \(\varphi\left(R^{(1)}\right)=C^{(j)}\), we
  argue as above for \((\ind\tau)\circ\varphi\), and conclude the result with
  \(\psi=\tau\ppi\).
\end{proof}

\section{Max closed cones}\label{sec:max-cones}

Here we have a glimpse on cones that are also \op-monoids,
i.e. \emph{\maxclosed cones}, restricted to cones contained in
\(\Q_+^N\).  As we have two additive monoids on \(\Q_+^N\), we label
by the \op symbol all notions like submonoid and homomorphism to make
it clear which structure we are referring to. While the restriction to
rational instead of real cones appears out of the blue here,
\cref{rat-cones} gives a rationale for that.

\begin{pro}\label{prop-4.:1}
  A rational cone \CC  is a \op-submonoid of \(\Q_+^N\)
  if and only if its integer cone \(\CC_\Z =\CC\cap\ZN\) is a
  \op-submonoid of \NN.
\end{pro}
\begin{proof}
  Only the  `if' part requires a proof. Let
  \(u,v\in \CC\).  Choose a positive integer \(r\) such that \(ru,rv\in\NN\).
  Clearly those vectors are in \CC, hence so is \(ru\oplus rv=r(u\oplus
  v)\). The last equality shows that \(u\oplus v\in\CC\), as required.
\end{proof}

\begin{pro}\label{mc-hyp}
  Let \(H\) be a half-space given by a linear inequality \(ax\geq 0\).
  Then,
  \begin{enumerate}
      \item If \(a\) has at most one negative component, \(H\) is \maxclosed.
      \item If \(a\) has at least two negative components, then any small
    neighborhood in the bounding hyperplane of \(H\) contains points whose
    \(\max\) is outside \(H\).
  \end{enumerate}
\end{pro}
\begin{proof}
  If \(a\) is nonnegative, it is clear that \(H\) is \maxclosed. Suppose it
  has a single negative component; we can rewrite the inequality as
  \(cx-bx_i\geq 0\), where \(c\geq 0, c_i=0, b>0\). Let \(u,v\in H\); Without
  loss of generality, \(\max(u_i,v_i)=v_i\).  Then,
  \(c(u\op v)\geq cv \geq bv_i = b(u\op v)_i\), showing that \(u\op v\in H\).
  This shows part (a).

  Suppose that there are distinct \(r,s\in N\) such that \(a_r, a_s<0\).  Let
  \(u\) be a point on the hyperplane \(ax=0\) and choose any
  \(\varepsilon>0\). Let \(z\in\NN\) have components \(z_r=-a_s\),
  \(z_s=a_r\), all other components being 0; then, \(az=0\), and both
  \(u+\varepsilon z\) and \(u-\varepsilon z\) lie in the hyperplane. Let
  \(v=(u+\varepsilon z)\oplus (u-\varepsilon z)\); then,
  \(v_s=u_s-\varepsilon a_r\), \(v_r=u_r-\varepsilon a_s\), and \(v_i=u_i\)
  otherwise.  But \(av=au -2\varepsilon a_r a_s<0\), so \(v\not\in H\).
\end{proof}

\begin{teo}\label{op-closed}
  A full dimensional cone is \maxclosed if only if for every facet
  inequality \(ax\geq 0\), \(a\) has at most one negative component.
\end{teo}
\begin{proof}
  One one hand, if the inequalities are of the form given, the cone is an
  intersection of \maxclosed half-spaces, hence \maxclosed.

  If a facet inequality of a cone is not of the specified form, then, applying
  Prop. \ref{mc-hyp} to a neighborhood of an interior point of that facet,
  we see that the cone is not \maxclosed.
\end{proof}

\begin{cor}\label{max-closed}
  A full dimensional nonnegative cone is \maxclosed if
  only if every facet inequality is either of form \(x_j\geq 0\) or
  \(ax\geq 0\), where \(a\) has exactly one negative component and at least
  one positive component.
\end{cor}
\begin{proof}
  Since each \(x_j\geq 0\) is a valid inequality for \CC, all facet
  inequalities \(ax\geq 0\) not of this type must have at least one negative
  coefficient, say, \(a_i<0\).  By \cref{op-closed}, it is exactly one and
  there must be a \(j\) such that \(a_j>0\), otherwise, any \(x\in\CC\) would
  satisfy \(x_i\leq 0\) and be nonnegative, whence, \(x_i=0\), contradicting
  full dimension.
\end{proof}

\section{\op-automorphisms}\label{sec:op-automorphisms}

In \RN, \(x\leq y\) if and only if \(x\op y=y\), so any
\op-automorphism also preserves \(\leq\).  We recall that in a partial
order, \(y\) \emph{covers} \(x\) if \(x\lneq y\) and there is no third
element \(z\) such that \(x\lneq z \lneq y\).  In \ZN, \(y\) covers
\(x\) if and only if \(y=x+e_i\), for some \(i\).  When dealing with
integer vectors, we will use interval notation to refer
implicitly to \ZN, that is \([x,y]=\conj{z\in\ZN}{x\leq z\leq y}\).

We will be concerned here with \op-automorphisms of integer \maxclosed
cones only.  On one hand, this is motivated by our interest in
exponential matrices; on the other hand, \cref{rat-cones} presents a
brief discussion on real and rational cones in this context, and their
difficulties.  Still, some facts that help tame additive automorphisms
of integer \maxclosed cones do not hold for \op-automorphisms, and we
will need a few additional hypotheses on the cone.

\begin{example}\label{op-nonrespecting}
  Here we show a family of \maxclosed cones, and a \op-automorphism of
  the corresponding integer cones which cannot be extended to an
  additive automorphism.  An intuitive geometric explanation for that
  is that the cones are ``too thin''.  For any positive integer \(k\),
  let \(\CC_k\) be the 2-dimensional cone given by:
  \[
    \biggl\{\begin{array}[c]{rcrl}
              -k\,x  & + & (k+1)y & \geq 0   \\
              (k+1)x & - & k\,y   & \geq 0 .
            \end{array}
  \]
  One readily verifies that \(\CC_k\) is symmetric about the line \(x-y=0\).

  The points \(p=(k+1,k)\) and \(q=(k,k+1)\) are special here: in the lattice
  \(\CC_k\cap\Z^2\) each of them covers only \((k,k)\) and is covered only by
  \((k+1,k+1)\).  To see this, suppose \((x,y)\in \CC_k\cap\Z^2\backslash\{p,q\}\)
  satisfies \((x,y)\geq p\).  We want to show that \((x,y)\geq (k+1,k+1)\); if
  \(y>k\), then \(y\geq k+1\) and we are done, and the case \(y=k\) would require
  \(x>k+1\), which is ruled out by the first defining inequality.  All remaining
  verifications are similar.

  It follows that the involution on \(\CC_k\cap\Z^2\) that interchanges \(p\) and
  \(q\) and fixes all other points is order preserving - hence a \op-automorphism
  of the integer cone.

  Since this map moves \(p\) and fixes \(2p\), it cannot be extended to an
  additive map.
\end{example}

We say that an arbitrary \(S\subseteq\RN\) is \emph{very full} if, for
every \iinN, \(\one\pm e_i\in S\) (recall that \one denotes the vector
of all 1's).

\begin{pro}\label{cone-very-full}
	A   cone is very full if and
	only if for every facet inequality \(ax\geq 0\), one has that
	\(a\one\geq\max_\iinN |a_i|\).
\end{pro}
\begin{proof}
	Consider a vector \(w=\one+\alpha e_i\), where \(\alpha=\pm 1\). Then,
	\(aw\geq 0\) if and only if \(a\one\geq -\alpha a_i\), and that
	happens if
	and only if \(a\one\geq |a_i|\). 
\end{proof}

If \CC is a subset of \RN, and \phi is a bijective map from \CC to
itself, we will say that \phi is \emph{permutational} if there
exists a permutation \(\pi\) of \(N\) such that for every
\(a=(a_i)\in\CC\), \(\phi(a)_{\pi(i)}=a_i\) for all \(i\in N\).
That means that \phi is the restriction to \CC of the linear map
whose action on the canonical basis \((e_i)_\iinN\) is given by
\(e_i\mapsto e_{\pi(i)}\).

Note that any permutational map is an additive homomorphism, a
\op-homomorphism and an isometry.

\begin{teo}\label{pro-veryfull}
	Let \CC be a very full subset of \NN closed under \(+\) and \op.
	Then, every
	\op-automorphism fixing a non-zero multiple of \one is permutational.
\end{teo}
\begin{proof}
  Notice that \(\one = (\one -e_1)\oplus (\one-e_2)\in \CC\), since
  \(\one\pm e_i\in\CC\). Furthermore, for every integer \(k>0\),
  \iinN, \(k\one\pm e_i=(k-1)\one+\one\pm e_i\in\CC\). It will be
  convenient to denote \(b_i^k=k\one-e_i\), and
  \(B^k=\conj{b_i^k}{\iinN}\).

	Let \phi be a \op-automorphism fixing \(r\one\).

	We will proceed through a series of claims.
	\vspace{-1ex}
	\begin{description}
		\item [\emph{Claim 1}] For every integer \(k\geq 0\), \phi
		      fixes \(k\one\).\\
		      It is enough to prove that, for \(k>0\), if \phi fixes \(k\one\),
		      then it fixes both \((k+1)\one\) and \((k-1)\one\), and then
		      proceed by induction up and down, starting from \(r\one\).

		      The set of vectors covering \(k\one\) is invariant under \phi.
		      Those are \(\{k\one+e_i\}_\iinN\), and so
		      \((k+1)\one=\oplus_\iinN(k\one+e_i)\) is fixed by \phi. Also,  the
		      set  \(B^k\) of vectors covered by \(k\one\) is invariant, so
		      \(\op\conj{x\in\CC}{x\leq b_i^k,\,\text{for all
				      \iinN}}=(k-1)\one\) is also fixed, and the claim is proved.
	\end{description}

	Notice that for every integer $k\geq 1$, the set \(B^k\) consists
	of the coatoms of the interval \([0,k\one]\) in \ZN, and,
	\textsl{a fortiori}, in \CC, so, by Claim 1, \(B^k\) is invariant
	under \phi. Let \(\pi\in S_N\) be the permutation defined by
	\(\phi(b_i^1)=b_{\pi(i)}^1\).

	\begin{description}
            \item [\emph{Claim 2}] For every \(k\geq 1\), \(i\in N\),
          \(\phi(kb_i^1)=kb_{\pi(i)}^1\).\\
          Fix a \(k>1\).  We know already from the proof of Claim 1,
          that $\phi$ permutes the vectors from \(B^k\). Let us show
          $\phi$ also permutes \((kb_i^1)_\iinN\).  This set is
          precisely
          \[\conj{w\in\CC}{|[0,w]\cap B^1|=1=|[w,k\one]\cap B^k|},\]
          which shows our set is invariant under \phi.  Since one must
          have \(\phi(b_i^1)\leq\phi(kb_i^1)\), the claim follows.

            \item [\emph{Claim 3}] For every \(k \geq 1\),\(i\in N\),
          the interval \([kb_i^1,k\one]\) in \CC is a chain of
          height $k$ (i.e. length \( k+1\)).\\
          Clearly this interval consists of the vectors
          \(kb_i^1\op t\one\), \(0\leq t\leq k\), which gives the
          claim.
	\end{description}
	Now we finish the proof.  Given any \(v\in\CC\), we want to show that
	\(\phi(v)_{\pi(i)}=v_i\), for each \(i\).  Choose \(k\) bigger than any
	component of \(v\), and of \(\phi(v)\). The vector \(kb_i^1\op v\) has all
	components \(k\), except for the \(i^\text{th}\), which equals \(v_i\).  So,
	the interval \([kb_i^1,kb_i^1\op v]\) is a chain of height \(v_i\). It is
	mapped bijectively by \phi to \([kb_{\pi(i)}^1,kb_{\pi(i)}^1\op \phi(v)]\),
	which is a chain of height \(\phi(v)_{\pi(i)}\).  It follows that
	\(\phi(v)_{\pi(i)}=v_i\), as claimed.
\end{proof}

The examples below illustrate the precision of \cref{pro-veryfull}.


\begin{example}
	The set \CC needs not be an integer polyhedral cone: fix a real
	\(0<\alpha< 1\), and for \(n\geq 3\), let
	\(\CC=\conj{x\in\N^n}{\sum_{i\neq j}x_i\geq x_j^\alpha, j=1,\ldots,n}\).
	Since $x^\alpha +y^\alpha \geq (x+y)^\alpha$ for all $x, y\in \R_+$, this is
	an additive submonoid of \(\N^n\), and it is clearly very full. Following
	the proof of \cref{mc-hyp}, we see that \CC is also closed under \op.
	The elements of \CC with minimal support are all vectors with exactly two
	ones and zeros elsewhere; any \op-automorphism keeps this set invariant,
	hence it fixes \one, its max. \cref{pro-veryfull} implies that every
	\op-automorphism is permutational (by the description symmetry, any
	permutation of coordinates yields a \op-automorphism). On the other hand, it
	is an exercise to show that the smallest polyhedral cone containing \CC is the
	positive orthant, which is, of course, too large, as \(e_1\notin\CC\).
\end{example}

\begin{example}
	Here we show that the requirement of being an additive monoid cannot be
	simply dismissed.  For each \(k\in\N\), let
	\(D_k=\conj{v\in\NN}{k\one\leq v\leq(k+1)\one}\), and let
	\(D=\cup_{k\in\N}D_k\). Then \(D\) is \op-submonoid of \NN, and very full,
	but it is not closed under \(+\).  Choose, for each \(k\), a permutation
	\(\pi_k\in S_N\), and let \(T\) be the map that acts like \(\pi_k\) on the
	layer \(D_k\).  This is a \op-automorphism, but, unless all the \(\pi_k\)
	are equal, it is not permutational.
\end{example}

\begin{example}\label{rat-cones}
  If \CC is a real, full dimensional \op-closed cone, then every
  \op-automorphism is fully determined by its action on
  \(\CC\cap\QN\).  That is because each \(x\in\CC\) is the greatest
  lower bound of the set \(\conj{y\in\CC\cap\QN}{x\leq y}\).  On the
  other hand, the rational cone \Qpn admits a quite complicated
  automorphism group: choose on each component an increasing function
  on \(\Q_+\); the whole choice can even be done so as to fix \NN.
  Actually, in the same vein as in \cite{AS}, \cref{op-auto-veryfull}
  is likely also true for rational and real cones, \emph{except} for
  orthants, but we do not pursue this here.
\end{example}

\begin{lemma}\label{op-fix-one}
	Let \CC be a \op-submonoid of \NN. Assume also that \CC is fixed by a group
	of permutational maps that is transitive on the canonical
	basis.  Then every \op-automorphism of \CC fixes a multiple of \one.
\end{lemma}
\begin{proof}
	By Dickson's Lemma \cite{Di}, \cite{Go}, \CC has a finite set of
	minimal non-zero vectors; this set is invariant under any
	\op-automorphism of \CC, so its \op-sum is fixed by those
	automorphisms.  That is a non-zero vector fixed by a
	transitive permutation group, so it is a positive multiple of \one.
\end{proof}

Combining \cref{pro-veryfull} with
\cref{op-fix-one}, we obtain:

\begin{cor}\label{op-auto-veryfull}
	Let \CC be a very full subset of \NN closed under \(+\) and \op. Assume also
	that \CC is fixed by a group of permutational maps that is transitive on the
	canonical basis.  Then every \op-automorphism of \CC is permutational.
\end{cor}

Let us apply this now to \En.  This is very full, closed under \(+\) and \op,
and \(\Sn\subseteq\Aut(\En,\op)\) is a group as required by
\cref{op-auto-veryfull}, so every \op-automorphism of \En is permutational.
As permutational maps are linear, we have that
\(\Aut(\En,\op)\subseteq \LinZ(\En)=\Sn\), and we have proved:

\begin{teo}\mbox{\rm\cite[Theorem 4.3]{DKKP}}
	\(\Aut(\En,\op)=\Sn.\)
\end{teo}

Notice that \haten, as a rational polyhedral cone, satisfies the
following properties: it is very full, non-negative, pointed,
\maxclosed and each non-negativity inequality determines a
facet. The next example shows that even all of these
properties of a polyhedral cone are not enough to guarantee that each
additive automorphism is also a \op-automorphism.

\begin{example}
	Consider the cone
	\(\CC=\conj{x\in\Q^3}{x_1 +x_2\geq x_3, x_1,x_2,x_3\geq 0}\). The following
	hold:
	\begin{enumerate}
		\item \CC is very full, pointed, non-negative, \maxclosed.
		\item Each nonnegativity inequality of \CC determines a facet.
		\item There is an additive automorphism of \(\CC_\Z\) which is not permutational and
		      does not preserve \op.
	\end{enumerate}

	Clearly (a) is satisfied. Fact (b) is proved by the respective interior
	points \((0,2,1),(2,0,1),(1,1,0)\).

	The extreme rays of \CC are
	\[
		v_1 =(1, 0, 0), v_2 =(0, 1, 0), v_3 =(1, 0, 1), v_4 =(0, 1, 1).
	\]

	Let \(\psi\in GL(3,\Z)\) be given by the matrix
	\[
		\left(
		\begin{array}{rrr}
				1 & 1 & -1 \\
				0 & 0 & 1  \\
				0 & 1 & 0
			\end{array}
		\right)
	\]

	Then \(\psi\) fixes \(v_1\) and \(v_4\) and interchanges \(v_2\) and
	\(v_3\).  As it permutes extreme rays, it leaves \CC invariant, and is an
	additive automorphism of \(\CC_\Z\).  On the other hand,
	\(\psi(v_1\op v_3)\neq\psi(v_1)\op\psi(v_3)\), showing (c).
\end{example}

\bibliography{expmatr2}{}
\bibliographystyle{plainnat}

\end{document}